\documentclass[12pt]{article}

\usepackage{geometry}
\usepackage{color}
\usepackage{float}
\usepackage{textcomp}
\usepackage{amsthm}
\usepackage{amsmath}
\usepackage{amssymb}

\usepackage{changes}

\usepackage{mathtools}

\newtheorem{theorem}{Theorem}[section]
\newtheorem{lemma}[theorem]{Lemma}
\newtheorem{corollary}[theorem]{Corollary}
\newtheorem{proposition}[theorem]{Proposition}

\newtheorem{definition}[theorem]{Definition}

\DeclareMathOperator{\argmin}{argmin}

\DeclareMathOperator{\dom}{dom}

\newcommand{\R}{\mathbb{R}}

\newcommand{\inner}[2]{\langle{#1},{#2}\rangle}

\newcommand{\cP}{{\cal{P}}}

\newcommand{\vgap}{\vspace{.1in}}

\newcommand{\bi}{\begin{itemize}}
\newcommand{\ei}{\end{itemize}}
\newcommand{\ba}{\begin{array}}
\newcommand{\ea}{\end{array}}

\begin{document}

\title{Convergence rate bounds for a proximal ADMM with over-relaxation stepsize parameter for solving \\
nonconvex linearly constrained problems}

\author{
    Max L.N. Gon\c calves
     \thanks{Instituto de Matem\'atica e Estat\'istica, Universidade Federal de Goi\'as, Campus II- Caixa
    Postal 131, CEP 74001-970, Goi\^ania-GO, Brazil. (E-mails: {\tt
       maxlng@ufg.br} and {\tt jefferson@ufg.br}).  The work of these authors was
    supported in part by  CNPq Grants 406250/2013-8, 444134/2014-0 and 309370/2014-0.}
    \and
      Jefferson G. Melo \footnotemark[1]
    \and
    Renato D.C. Monteiro
    \thanks{School of Industrial and Systems
    Engineering, Georgia Institute of
    Technology, Atlanta, GA, 30332-0205.
    (email: {\tt monteiro@isye.gatech.edu}). The work of this author
    was partially supported by NSF Grant CMMI-1300221.}
}

\date{February 3, 2017 (Revised: November 02, 2017)}

\maketitle

\begin{abstract}
This paper establishes convergence rate bounds for a variant of the proximal alternating direction method of multipliers (ADMM) for
 solving nonconvex linearly constrained optimization problems.
The variant of the proximal ADMM allows the inclusion of an over-relaxation stepsize parameter belonging to the interval $(0,2)$. 
To the best of our knowledge, all related papers in the literature only consider the case where the over-relaxation parameter lies in the interval
$(0,(1+\sqrt{5})/2)$.   
 \end{abstract}

\pagestyle{plain}
\section{Introduction} \label{sec:int}
We consider  the following linearly constrained  problem 
\begin{equation} \label{optl}
\min \{ f(x) + g(y) : A x + B y =b, \; x\in \R^n, y \in \R^p \}
\end{equation}
where  $f: \R^{n} \to (-\infty,\infty]$ and $g: \R^{p} \to (-\infty,\infty]$ are proper lower semicontinuous functions,
$A\in \R^{l \times n}$,  $B \in \R^{l\times p}$ and  $b \in \R^l$. 
Optimization problems such as  \eqref{optl} appear in many important applications such as  nonnegative matrix factorization, distributed matrix factorization, distributed clustering, sparse zero variance discriminant analysis, tensor decomposition, and matrix completion, asset allocation (see, e.g., \cite{Ames2016,5712153,Liavas,Wen2013,Xu2012,Zhang2014,Zhang20101}).
Moreover, it has observed that (specific variants of) the alternating direction method of multipliers (ADMM) can tackle many of the instances
arising in these settings extremely well despite many of them being nonconvex.

A particular ADMM class for solving~\eqref{optl}, namely, the  proximal ADMM, recursively computes a sequence
$\{(s_k,y_k,x_k)\}$ as
\begin{align}
x_k &= \argmin_{x} \left \{  {\mathcal L}_\beta(x,y_{k-1},\lambda_{k-1})+\frac{1}{2}\|x- x_{k-1}\|_{G}^2 \right\},\nonumber\\
y_k &= \argmin_y \left \{ {\mathcal L}_\beta(x_k,y,\lambda_{k-1}) +\frac{1}{2}\|y- y_{k-1}\|_{H}^2\right\},\label{ADMMclass}\\
\lambda_k &= \lambda_{k-1}-\theta\beta\left[Ax_k+By_k-b\right] \nonumber
\end{align}
where $\beta>0$ is a penalty parameter, $\theta>0$ is a  stepsize parameter,
 $G \in \R^{n \times n}$ and  $H \in \R^{p \times p}$ are symmetric and positive semidefinite matrices, and 
 \[
{\mathcal L}_\beta(x,y,\lambda):=f(x)+g(y)-\langle \lambda,Ax+By-b\rangle+\frac{\beta}{2}\|Ax+By-b\|^2
\]
is  the augmented Lagrangian function for problem \eqref{optl}. If $(H,G)=(0,0)$ in the above method, we obtain the standard ADMM. 
 Moreover,   the above subproblems with suitable choices of $G$ and $H$ are easy to
solve or  even have  closed-form solutions for many relevant instances of \eqref{optl} (see   \cite{Deng1,HeLinear,Wang2012,Yang_linearizedaugmented} for more details).
 
For the case in which $f$ and $g$ in  \eqref{optl} are both convex (e.g., see \cite{MJR2,He2,HeLinear,monteiro2010iteration}),  the complexity results for the proximal ADMM~\eqref{ADMMclass} can be conveniently stated  in terms of the following simple termination criterion associated with the optimality condition for~\eqref{optl}, namely: for given $\rho, \varepsilon>0$, terminate with a quintuple
$(x,y, \lambda,r_1,r_2) \in \R^n \times \R^p \times \R^l\times \R^n\times \R^p$ satisfying
\begin{equation}\label{sttoping}
\max \{ \|Ax +By -b\|, \|r_1 \|, \|r_2\| \} \le \rho, \quad
r_1 \in \partial_{\varepsilon} f(x) - A^* \lambda, \quad r_2 \in \partial_{\varepsilon} g(y) - B^*  \lambda
\end{equation}
where $\partial_{\epsilon}$ denotes  the  classical $\epsilon$-subdifferential of convex functions and the norms in the first
inequality can be arbitrarily chosen.
In terms of this termination criterion, the best ergodic iteration-complexity bound found in the literature is
${\cal O}(\max\{\rho^{-1},\varepsilon^{-1}\})$ while the best pointwise
one is ${\cal O}(\rho^{-2})$. (The latter bound is independent of $\varepsilon$ since, in the pointwise case,
 the two inclusions above
are shown to hold with $\varepsilon=0$.)

This paper considers the special case of \eqref{optl} in which $f$ is as stated immediately following~\eqref{optl} (and hence not necessarily convex)
and $g$ is a differentiable function whose gradient is Lipschitz continuous on the whole $\R^p$.
By considering an extended notion of subdifferential for the nonconvex function $f$ (see for example \cite{Mordu2006,VariaAna}), this paper
establishes an ${\cal O}(\rho^{-2})$-pointwise iteration-complexity bound
to obtain
a quadruple 
$(x,y, \lambda,r_1) \in \R^n  \times \R^p\times \R^l\times \R^n$ satisfying
\[
\max \{ \|Ax +By -b\| , \|\nabla g(y) - B^* \lambda\|, \|r_1 \| \} \le \rho, \quad
r_1 \in \partial f(x) - A^*  \lambda.
\]
for an important subclass of the proximal  ADMM~\eqref{ADMMclass}. The latter subclass has the following properties:
the penalty parameter $\beta$ is sufficiently large (see \eqref{assump:betatau}), $G$ is an arbitrary positive semidefinite matrix,
$H$ is a sufficiently large positive multiple of the identity, and the stepsize $\theta$ lies in
the interval $(0,2)$.
To the best of our knowledge, no iteration-complexity has been established in the literature  for a variant of the ADMM with stepsize
$\theta > (\sqrt{5}+1)/2$, even for the case in which \eqref{optl} is assumed to be a  convex problem.
It is worth pointing out that  \cite{FPST_editor,glowinski1984} show that larger choice of $\theta$ usually
improves the practical performance of the proximal ADMM.
\\[2mm]
{\bf Previous related works.} 
The  ADMM was  introduced in \cite{0352.65034,0368.65053} and is thoroughly discussed in \cite{Boyd:2011,glowinski1984}.
Even though convergence of the sequence generated by the ADMM has been established in very early papers about it,
only recently has its iteration-complexity been established. To discuss this development in the convex case, we
 use the terminology
weak pointwise or strong pointwise bounds
to refer to complexity bounds relative to the best of the first $k$ iterates or the last iterate, respectively, to
satisfy the termination criterion \eqref{sttoping}.
The first iteration-complexity bound for the  ADMM  was established in  
\cite{monteiro2010iteration} under the assumptions that $C$ is injective.
More specifically, the ergodic iteration-complexity for the standard ADMM  is derived in  \cite{monteiro2010iteration}   for any $\theta \in (0,1]$ while
a weak pointwise iteration-complexity easily follows from the approach in \cite{monteiro2010iteration} for any $\theta \in (0,1)$.
Subsequently, without assuming that $C$ is injective, \cite{HeLinear} established the ergodic iteration-complexity of the proximal  ADMM~\eqref{ADMMclass}
with $G=0$ and $\theta=1$ and, as a consequence, of the  split inexact Uzawa method~\cite{Xahang}.
Paper \cite{He2} establishes the weak pointwise and ergodic iteration-complexity  
of another collection of ADMM instances which includes the standard ADMM for any $\theta \in (0,(1+\sqrt{5})/2)$.
It should be noted however that \cite{He2,HeLinear} do not provide any details on how to obtain an easily 
verifiable ergodic termination criterion with a well-established iteration-complexity bound.
A strong pointwise iteration-complexity bound for the proximal ADMM   \eqref{ADMMclass} with $G=0$ and $\theta =1$ is derived in~\cite{He2015}. 
Pointwise and ergodic iteration-complexity results for the whole proximal ADMM~\eqref{ADMMclass} and for any $\theta \in (0,(1+\sqrt{5})/2)$ are
given  in \cite{Cui,Gu2015}.
In addition to providing alternative proofs for these latter results, paper \cite{MJR2} obtains an ergodic iteration complexity
bound for the proximal ADMM with $\theta=(1+\sqrt{5})/2$.
Finally, a number of papers (see for example \cite{Deng1,GADMM2015,MJR,Hager,Lin,LanADMM} and  references therein) have obtained
similar  complexity results in the context of other ADMM classes.

Iteration-complexity analysis of the ADMM has also been established for possibly nonconvex instances of \eqref{optl} satisfying
the same assumptions made on this paper, i.e.,
$f$ is a proper lower semi-continuous
function  and $g$ is a continuously differentiable function whose gradient is Lipschitz continuous on the whole $\R^p$.
Recently, there have been a lot of interest on the study of ADMM variants for nonconvex problems (see, e.g., \cite{ADMM_KL,Hong20,Hong2016,Jiang2016,SplitMet_NonConv,Multi-blockBregman,BregmanADMM,wotao2015,Stepzisenoncon}).
The results developed in \cite{ADMM_KL,SplitMet_NonConv,Multi-blockBregman,BregmanADMM,wotao2015,Stepzisenoncon} establish
convergence of the generated sequence to a stationary point of \eqref{optl} under the assumption that  the objective function
of \eqref{optl} satisfies the so-called Kurdyka-Lojasiewicz (K-L) property. However, none of these papers considers the issue of iteration complexity
for ADMM although their theoretical analysis are generally half-way or close to accomplishing such goal.
Paper \cite{Hong2016} analyzes the convergence of ADMM for solving nonconvex consensus and sharing problems and establishes
the iteration complexity of ADMM for the consensus problem.
Paper \cite{Jiang2016} studies the iteration-complexity of a multi-block type ADMM method whose two-block
special case is a modification of the proximal ADMM in which the function $g$ of the second subproblem in \eqref{ADMMclass} is
replaced by its linear approximation, $G$ is positive definite and
$H$ is chosen as $LI$ where $L$ is the Lipschitz constant of $\nabla g(\cdot)$.
 Finally, \cite{Hong20} studies the iteration-complexity of a  proximal variant of the augmented Lagrangian method for
solving the $1$-block special form of \eqref{optl}, i.e., with $f=0$ and $A=0$.
\\[2mm]
{\bf Organization of the paper.}    Subsection~\ref{sec:bas}  presents some notation and basic results. Section~\ref{ProximalADMM}
describes the  proximal ADMM and presents corresponding convergence rate bounds whose proofs are given in Subsection~\ref{sec:hpe_Analysis}.

\subsection{Notation and basic results}
\label{sec:bas}
This subsection presents some definitions, notation and basic results used in this paper.

Let $\R^n$ denote the $n$-dimensional Euclidean space with inner product and associated norm denoted by $\inner{\cdot}{\cdot}$ and $\|\cdot\|$, respectively. We use $\R^{l\times n}$ to denote the set
of all $l\times n$ matrices.  
The image space of a matrix $Q\in \R^{l \times n}$ is defined as  ${\rm Im}(Q):=\{Qx: x \in \R^n\}$
and $\cP_Q$ denotes the Euclidean projection onto $\mbox{Im}\, (Q)$. The notation $Q\succ 0$ means that $Q$ is a definite positive matrix. 
The symbol  $\lambda_{\min}(Q)$   denotes the  minimum eigenvalue of a symmetric matrix $Q$.
If $Q$ is a symmetric and positive semidefinite matrix, the seminorm induced by $Q$ on $\R^n$, denoted by $\|\cdot\|_{Q}$, 
is defined as $\|\cdot\|_{Q}= \langle Q (\cdot), \cdot\rangle ^{1/2}$. For a given sequence $\{z_k: k \ge 0\}$,  let  $\{\Delta z_k\}$ be the sequence
defined by
 $$\Delta z_k:=z_k-z_{k-1}, \quad k\geq 1.$$ 
The domain of a function
 $h :\mathbb{R}^n\to (-\infty,\infty]$ is the set   $\dom h := \{x\in \R^n : h(x) < +\infty\}$. 
Moreover, $h$ is said to be proper if  $h(x) < \infty$ for some $x \in \R^n$.

%

We next recall some definitions and results of  subdifferential calculus~\cite{Mordu2006,VariaAna}.

\begin{definition}
Let $h: \R^n \to(-\infty,\infty]$ be a proper lower semi-continuous function.
\begin{enumerate}
\item[(i)] The Fr\'echet subdifferential of $h$ at $x\in \dom h$, written by $\hat{\partial} h(x)$, is the set of all elements $u \in \R^n$ which satisfy
$$\liminf_{y\neq x\; y\to x} \frac{h(y)-h(x)-\inner{u}{y-x}}{\|y-x\|}\geq0.$$
When $x\notin \dom h$, we set $\hat{\partial} h(x)=\emptyset$.
\item[(ii)] The limiting subdifferential, or simply subdifferential, of $h$ at $x \in dom\, h$, written by ${\partial} h(x)$, is defined as 
\[
\partial h (x)=\{u\in \R^n:\exists \,  x_n\to x, h(x_n)\to h(x), u_k \in \hat{\partial} h(x_n), \;\mbox{with}\; u_k\to u \}.
\]
\item[(iii)]  A critical (or stationary) point of $h$ is a point $x$ in the domain of $h$ satisfying $0\in \partial h(x)$.
\end{enumerate}
\end{definition}  
The following result gives some properties of the subdifferential.
 
\begin{proposition} Let $h: \R^n \to (-\infty,\infty]$ be a proper lower semi-continuous function.
\begin{enumerate}
\item[(a)]  if $\{(u_k,x_k)\}$ is a  sequence such that $x_k \to x$, $u_k \to u$, $h(x_k)\to h(x)$ and $u_k \in  \partial h(x_k)$,
 then $u \in \partial h(x)$;
\item[(b)]  if $x \in R^n$ is a local minimizer of $h$, then $0\in  \partial h(x)$;
\item[(c)] if $p: \R^n \to \R$ is a continuously differentiable function, then $\partial(h + p)(x) = \partial h(x) + \nabla p(x)$.
\end{enumerate}
\end{proposition}

 We end this section by recalling the definition of  critical points of \eqref{optl}.
 \begin{definition}
 A triple $ (x^*,y^*, \lambda^*) \in \R^n\times \R^p\times  \R^l$ is a critical point of problem~\eqref{optl}
if
\[ 0\in  \partial f(x^*)-A^*\lambda^*, \quad 0=  \nabla g(y^*)-B^*\lambda^*,
\quad 0=  Ax^*+By^*-b.
\]
 \end{definition}
Under some mild conditions, it can be shown that if $(x^*,y^*)$ is a local minimum of \eqref{optl}, then there exists
$\lambda^*$ such that $ (x^*,y^*, \lambda^*)$ is a critical point of \eqref{optl}.
\section{Proximal ADMM and its convergence rate}\label{ProximalADMM}

This section describes the assumptions made on problem \eqref{optl} and states the variant of the proximal ADMM considered in this paper.
It also states the main result of this paper  (Theorem~\ref{maintheo}), and a special case of it (Corollary \ref{cor:maincor1}),  both of them
describing convergence rate bounds for the aforementioned proximal ADMM variant.
The proof of Theorem~\ref{maintheo} is however postponed to Section \ref{sec:hpe_Analysis}.

The augmented Lagrangian associated with problem \eqref{optl} is defined as
\begin{equation}\label{lagrangian2}
{\mathcal L}_\beta(x,y,\lambda):=f(x)+g(y)-\langle \lambda,Ax+By-b\rangle+\frac{\beta}{2}\|Ax+By-b\|^2.
\end{equation}

This paper considers problem \eqref{optl} under the following set of assumptions:

\begin{itemize}
\item[{\bf(A0)}] $f: \R^{n} \to (-\infty,\infty]$ is a proper lower semi-continuous function;
\item[{\bf(A1)}] $B \ne 0$ and $ {\rm Im} (B) \supset \{b\} \cup {\rm Im} (A)$;

\item[{\bf(A2)}]
$g: \R^p \to \R$ is differentiable everywhere on $\R^p$ and there exists $L > 0$ such that
\[
\| \cP_{B^*}(\nabla g(y')) -  \cP_{B^*}(\nabla g(y)) \| \le L \|y'-y\| \quad \forall y,y' \in \R^p;
\]
\item[{\bf(A3)}]
there exists $m \ge 0$ such that the function $g(\cdot)+ m \|\cdot\|^2/2$ is convex, or equivalently,
\[
g(y') - g(y) - \inner{\nabla g(y)}{y'-y} \ge - \frac{m}2 \|y'-y \|^2 \quad \forall y,y' \in \R^p;
\]
\item[{\bf(A4)}] 
there exists $\bar \beta \ge 0$ such that
\[
\bar {\mathcal{L}}:=\inf_{(x,y)} \left \{ f(x)+g(y) +\frac{\bar \beta}{2} \| Ax+ By -b \|^2 \right\}>-\infty. \label{assump:beta_inf}
\]
\end{itemize}

Some comments are in order. First,  due to the generality of ({\bf A0}), problem \eqref{optl} may include
an extra constraint of the form $x \in X$ where $X$ is a closed set since this constraint can be incorporated
into $f$ by adding to it the indicator function of $X$.
Second, ({\bf A1})  implies that for every $x \in \R^n$, there exists $y \in  \R^p$ such that
$(x,y)$ satisfies the (linear) constraint of \eqref{optl}. The extra condition that
$B \ne 0$ is very mild since otherwise \eqref{optl} would be much simpler to solve.
Third,  if $\nabla g(\cdot)$ is $L$-Lipschitz continuous, then ({\bf A2})
and ({\bf A3}) with $m=L$ obviously hold.
However, conditions ({\bf A2}) and ({\bf A3}) combined are generally weaker than the condition that
$\nabla g(\cdot)$ be $L$-Lipschitz continuous.

Next we  state the   proximal ADMM  for solving  problem \eqref{optl}.
\\
\hrule
\noindent
\\
{\bf Proximal ADMM}
\\
\hrule
\begin{itemize}
\item[(0)] Let an initial point $(x_0,y_0,\lambda_0) \in \R^n\times \R^p\times  \R^l$ and  a symmetric positive semi-definite matrix $G\in \R^{n\times n}$ be given. 
Let a  stepsize parameter $\theta\in (0,2)$ be given and define
\begin{equation}\label{def:gamma}
\gamma:=\frac{\theta}{(1-|\theta-1|)^2}.
\end{equation}
Choose scalars $\beta\geq \bar{\beta}$ (see {\bf (A4)}) and $\tau\geq0$ such that   
\begin{align}\label{assump:betatau}
\delta_1 :=\left(\frac{\beta\sigma_B+\tau-m}{4}-\frac{3\gamma(L^2+\tau^2)}{\beta\sigma_B^+}\right)>0,
\end{align}
where $\sigma_B$ (resp., $\sigma_B^+$) denotes the smallest  eigenvalue (resp., positive eigenvalue) of $B^*B$, and set $k=1$;
\item[(1)]    compute an optimal solution $x_k \in \R^n$  of the subproblem
\begin{equation} \label{sub:x}
\min_{x \in \R^n} \left \{ {\mathcal L}_\beta(x,y_{k-1},\lambda_{k-1})+\frac{1}{2}\|x- x_{k-1}\|_{G}^2\right\}
\end{equation}
and then compute an optimal solution $y_k\in \R^p$ of the subproblem
\begin{equation} \label{sub:y}
\min_{y \in \R^p} \left \{ {\mathcal L}_\beta(x_k,y,\lambda_{k-1}) +\frac{\tau}{2}\|y- y_{k-1}\|^2\right\};
\end{equation}
\item[(2)] set 
\begin{equation}\label{eq:lambda}
\lambda_k = \lambda_{k-1}-\theta\beta\left[Ax_k+By_k-b\right]
\end{equation}
and $k \leftarrow k+1$, and go to step~(1).
\end{itemize}
{\bf end}
\\
\hrule
\vgap

We now make a few remarks about the proximal ADMM. 
First, the assumption that $\theta \in (0,2)$ guarantees that $\gamma$ in \eqref{def:gamma} is well-defined
and  positive.
Second, the special case of the proximal ADMM in which $G=0$  requires only an initial pair
$(y_0,\lambda_0)$ since any of its iteration is independent of $x_{k-1}$.
{Third,  inequality \eqref{assump:betatau} implies that $\beta B^*B + \tau I-m I \succ 0$.
Thus, the
objective function of subproblem \eqref{sub:y} is strongly convex and hence  $y_k$ is uniquely determined.}
Fourth, the subproblems  \eqref{sub:x} and \eqref{sub:y} are of the form
\[
\min_{x \in \R^n}\left\{ f(x) + \inner{c}{x}+\frac{1}{2}\|x\|_{G+\beta A^*A}^2 \right\}, \quad
\min_{y \in \R^p}\left\{ g(y)+  \inner{d}{y} +\frac{1}{2}\|y\|_{\tau I + \beta B^*B}^2 \right\}
\]
for some $c \in \R^n$ and $d \in \R^p$. For the purpose of this paper, we assume they are easy to solve exactly, possibly by
choosing $\tau\geq0$, $\beta>0$ and  $G$ appropriately.
Fifth, condition~\eqref{assump:betatau} imposed on the different data constants and parameters of the proximal ADMM method
are needed to establish convergence rate bounds for it (see Theorem~\ref{maintheo}).
Note that, if either $\sigma_\beta>0$ or $\tau>0$, then it is always possible to choose a sufficiently large
penalty parameter $\beta$ satisfying this condition.
Hence, it is possible to obtain convergence rate bounds for the standard ADMM (i.e., the special case of
the above method with $G=0$ and $\tau=0$) for $\beta$ sufficiently large (see Corollary \ref{cor:maincor1}).


Next we define a  parameter required in order to present our convergence rate bounds.
Define
\begin{align}
\eta_0 (y_0,\lambda_0;\theta) := \min_{(\Delta y_0, \Delta \lambda_0) } &\; \frac{c_1}{2}\|B^*\Delta\lambda_0\|^2
+\left(\frac{\beta\sigma_B+\tau-m}{4}\right)\|\Delta y_0\|^2 \nonumber \\
 {\rm s.t.}  \ \ & \tau\Delta y_0+(1-1/\theta)B^*\Delta \lambda_0=  B^*\lambda_0 -\nabla g(y_0) \label{def:eta_00}
\end{align}
where
\begin{equation}\label{def:c1} 
 c_1:=\frac{2|\theta-1|}{\beta\theta(1-|\theta-1|)\sigma_B^+} \geq0.
\end{equation}

Theorem \ref{maintheo} below  expresses the complexity of the proximal ADMM in terms of the quantity $\eta_0$, 
which depends on the initial iterate pair  $(y_0,\lambda_0)$ as well  as the constant $m$ and the parameters $\theta$, $\beta$ and
$\tau$ used by the method. This contrasts with the analysis of the papers \cite{Hong20,Hong2016,Jiang2016} which derive iteration-complexity
for variants of the augmented Lagrangian and
the proximal ADMM expressed in terms of both $(x_0,y_0,\lambda_0)$ and $(x_1,y_1,\lambda_1)$.
We believe that the one derived in this paper is more convenient since quantities expressed only in terms of
$(x_0,y_0,\lambda_0)$ are easier to compute and/or estimate. Definition \eqref{def:eta_00} of $\eta_0$ is somewhat complicated
but, under some conditions, it simplifies or an upper bound on $\eta_0$ can easily be obtained.
The following trivial result elaborates  on this point and
gives sufficient conditions for the quantity  $\eta_0$  to be finite.


\begin{lemma}\label{lem:125} Let $(y_0,\lambda_0) \in \R^p\times  \R^l$ and $\theta\in (0,2)$ be given.
Then, problem \eqref{def:eta_00} is feasible, and hence the quantity $\eta_0 := \eta_0(y_0,\lambda_0;\theta)$
 is finite, under either one of the  following conditions:
\begin{itemize}
\item[(i)]  $\tau=0$ and $B^*\lambda_0=\nabla g(y_0)$, in which case $\eta_0=0$;
\item[(ii)] $\tau=0$, $\theta\neq1$ and  $B^*B$ invertible;
\item[(iii)] $\tau>0$.
\end{itemize}
\end{lemma}

We note that convergence rate bounds for the proximal ADMM have been derived
in the convex setting  whenever $\beta>0$ and $\theta \in (0,(1+\sqrt{5})/2)$ (see for example \cite{MJR2}).
However, derivation of similar bounds for the case in which $\theta \ge (1+ \sqrt{5})/2$ is not known even in the convex setting.
The following result derives  convergence rate bounds for the proximal ADMM for solving
the nonconvex optimization problem \eqref{optl} satisfying  assumptions {\bf(A0)}-{\bf(A4)} for any $\theta \in (0,2)$ and $\beta$ sufficiently large.


\begin{theorem}\label{maintheo} 
Assume that the stepsize $\theta \in (0,2)$ and the initial pair $(y_0, \lambda_0 ) \in \R^n \times \R^p \times \R^l$ is
such that  the quantity $\eta_0 := \eta_0(y_0,\lambda_0;\theta)$ defined in \eqref{def:eta_00} is finite
and define
\begin{equation}\label{def:Ltilde} 
 \Delta^0_\beta:={\mathcal L}_\beta(x_0,y_0,\lambda_0)-{\bar{\mathcal L}}
\end{equation}
where  ${\bar{\mathcal L}}$ is as in {\bf (A4)}.
If, for every $k\geq 1$, we define
\begin{equation}\label{def:lambdahat}
\hat\lambda_k:=\lambda_{k-1}-\beta\left(Ax_k+By_{k-1}-b\right),
\end{equation}
then we have
\begin{equation} \label{eq:optimi.f}
-G\Delta x_k \in \partial f(x_k)-A^* \hat \lambda_k,
\end{equation}
and there exists $j\leq k$ such that

\begin{equation*} 
\|\Delta x_j\|_{G}\leq \sqrt{\frac{6\max\{\eta_0,\Delta^0_\beta\}}{k}}, \qquad
\|\nabla g(y_j)-B^*\hat \lambda_j\|\leq \left(\beta\|B^*B\|+\tau\right)\sqrt{\frac{3\max\{\eta_0,\Delta^0_\beta\}}{\delta_1k}},
\end{equation*}
\begin{equation*} 
\|Ax_j+By_j-b\|\leq \frac{1}{\beta\theta}\sqrt{\frac{3\max\{\eta_0,\Delta^0_\beta\}}{\delta_2k}} 
\end{equation*}
where $\delta_1$ is as in \eqref{assump:betatau}, and $\delta_2$ is defined as
\begin{equation}\label{def:deltas}
\delta_2:= \left(\beta\theta+\frac{6\theta\gamma(L^2+\tau^2)}{\sigma_B^+\delta_1}\right)^{-1}.
\end{equation}
\end{theorem}

As a consequence of the previous result, the following corollary establishes convergence rate bounds for the standard ADMM
for solving \eqref{optl} with invertible matrix $B$ for any stepsize $\theta \in (0,2)$ and sufficiently large penalty parameter $\beta$.


\begin{corollary}\label{cor:maincor1}
Consider the standard ADMM, i.e., the special case of the proximal ADMM with $G=0$ and $\tau=0$,
 applied to problem \eqref{optl} with coefficient matrix $B$  invertible.
Assume that the initial pair $(y_0,\lambda_0)$   satisfies $B^*\lambda_0 = \nabla g(y_0)$ and
$\beta \ge \bar \beta$ is chosen in a such a way that
\begin{equation}\label{eq:2353}
\frac{\beta\sigma_B-2m}8 \ge \frac{3\gamma L^2}{\beta\sigma_B}.
\end{equation}
Then, $\Delta^0_\beta \ge 0$  where   $ \Delta^0_\beta $ is as in \eqref{def:Ltilde}, and for every $k\geq 1$,
$$0 \in \partial f(x_k)-A^* \hat \lambda_k$$  and there exists $j\leq k$ such that
\begin{align*}
\|\nabla g(y_j)-B^*\hat \lambda_j\|\leq {\cal O} \left( \sqrt{\beta} \|B^*B\| \sqrt{\frac{ \Delta^0_\beta}{\sigma_B k}} \right),
\qquad
\|Ax_j+By_j-b\|\leq {\cal O} \left(\sqrt{\frac{\Delta^0_\beta}{\beta \theta k}}\right).
\end{align*}
\end{corollary}
\begin{proof} 
Note that the assumptions that $B$ is invertible and $B^*\lambda_0 = \nabla g(y_0)$, together with Lemma \ref{lem:125}(i), imply that
$\sigma_B=\sigma_B^+$ and $\eta_0=0$. 
The conclusion that $\Delta^0_\beta \ge 0$ follows from Lemma~\ref{lem:infL} with $k=0$,  and the fact  that $\eta_0=0$.
Moreover, inequality \eqref{eq:2353} yields 
$\gamma L^ 2\leq (\sigma_B\beta)^2/24$.
Hence, since $\tau=0$, 
it follows from the definitions of $\delta_1$ and $\delta_2$ in \eqref{assump:betatau} and \eqref{def:deltas}, respectively, and inequality  \eqref{eq:2353}   that
\[ 
\frac{\beta\sigma_B}8\leq\delta_1\leq \frac{\beta\sigma_B}4, \qquad \beta\theta\leq\frac{1}\delta_2\leq 3\beta\theta.
\]
Hence, $\delta_1={\cal O}(\beta\sigma_B)$ and $1/\delta_2={\cal O}(\beta\theta)$. Therefore, the  desired result trivially follows from
 the facts that $G=0$, $\tau=0$ and $\eta_0=0$, and Theorem~\ref{maintheo}. 
\end{proof}



\section{Proof of Theorem \ref{maintheo}}\label{sec:hpe_Analysis}

This section gives the proof of  Theorem~\ref{maintheo}.

We first establish a few technical lemmas. The first one describes a set of inclusions/equations satisfied by the sequence  $\{(x_k,y_k,\lambda_k)\}$ generated by the proximal ADMM.
 
 \begin{lemma} \label{pr:aux}
Consider the sequence  $\{(x_k,y_k,\lambda_k)\}$ generated by the proximal ADMM and let $\{\hat{\lambda}_k\}$ as defined in \eqref{def:lambdahat}. 
Then, for every $k\geq 1$, the following inclusions hold:
\begin{align}
0&\in \left[ \partial f(x_k)-A^*\hat\lambda_k\right] + G(x_k-x_{k-1}), \label{aux.0}\\[3mm]
0&= \left[\nabla g(y_k)-B^*\hat  \lambda_k\right]+\beta B^*B (y_k-y_{k-1}) + \tau(y_k-y_{k-1}),\label{aux.212}\\[3mm]
0&=  \left[Ax_k+By_k-b \right]+\frac{1}{\theta\beta}(\lambda_k-\lambda_{k-1}).\label{aux.32}
\end{align}
\end{lemma}
\begin{proof}
The optimality conditions for \eqref{sub:x} and \eqref{sub:y} imply that
\begin{align*}
0 &\in \partial f(x_k)- A^*({\lambda}_{k-1}-\beta(Ax_{k}+By_{k-1}-b))+G(x_k-x_{k-1}), \\
0 &= \nabla g(y_k)-B^*(\lambda_{k-1}-\beta(Ax_k+By_k-b))+\tau (y_k-y_{k-1}),
\end{align*}
respectively. These relations combined with \eqref{def:lambdahat} immediately yield  \eqref{aux.0} and \eqref{aux.212}.
Relation \eqref{aux.32} follows immediately from \eqref{eq:lambda}.
\end{proof}

The following lemma provides a recursive relation for the sequence $\{\Delta \lambda_k\}$.

\begin{lemma} \label{pr:aux-new}
Let $\Delta y_0\in \R^p$ and $\Delta \lambda_0\in \R^l$ be such that
 \begin{equation}\label{def:deltay0}
\tau\Delta y_0+(1-1/\theta)B^*\Delta \lambda_0=   B^*\lambda_0 -\nabla g(y_0).
 \end{equation}
Then, for every $k \ge 1$, we have
\begin{equation}\label{eq:B*lambda}
 B^*\Delta\lambda_k=(1-\theta)B^*\Delta\lambda_{k-1}+\theta u_k 
 \end{equation}
where 
\begin{equation}\label{def:uk}
u_k=\nabla g(y_k)-\nabla g(y_{k-1})+\tau(\Delta y_k-\Delta y_{k-1}).
\end{equation} 
\end{lemma}
\begin{proof}
Using \eqref{def:lambdahat} and \eqref{aux.32} we easily see that
$$
\theta \hat\lambda_k:=\lambda_{k}+(\theta-1)\lambda_{k-1} +\beta\theta B(y_k-y_{k-1}),\quad \forall k \ge 1.
$$
This expression together  with  \eqref{aux.212} then imply that
\begin{equation}\label{eq:auxlem1}
 B^* \lambda_k=(1-\theta)B^*\lambda_{k-1}+ \theta [ \nabla g(y_k) + \tau \Delta y_k] \quad \forall k \ge 1.
\end{equation}
Hence, in view of \eqref{def:uk},  relation \eqref{eq:B*lambda} holds for every $k \ge 2$.
Also, \eqref{def:uk} and \eqref{eq:auxlem1} both with $k=1$ imply that
\begin{align*}
B^* \Delta \lambda_1 &= B^* (\lambda_1-\lambda_0) = - \theta B^*\lambda_0 +  \theta \left [\nabla g(y_1) + \tau \Delta y_1\right ]
= - \theta B^*\lambda_0 +  \theta \left[ u_1 + \nabla g(y_0) +  \tau \Delta y_0 \right]
\end{align*}
which, together with the definition of $\Delta y_0$ in \eqref{def:deltay0}, shows that \eqref{eq:B*lambda} also holds for $k=1$.
\end{proof}

The next lemma describes how the sequence $\{(x_k,y_k,\lambda_k)\}$ affects  the value of
the augmented Lagrangian function defined in \eqref{lagrangian2}.

\begin{lemma} \label{lem:auxdecL} For every $k\geq 1$, we have
\begin{itemize} 
\item [(a)] ${\mathcal L}_\beta(x_k,y_{k-1},\lambda_{k-1})- {\mathcal L}_\beta(x_{k-1},y_{k-1},\lambda_{k-1})\leq -\|\Delta x_k\|_{G}^2/2$;
\item [(b)] ${\mathcal L}_\beta(x_k,y_k,\lambda_{k-1})- {\mathcal L}_\beta(x_k,y_{k-1},\lambda_{k-1})\leq (m-\beta\sigma_B-\tau)\|\Delta y_k\|^2/2$;
\item [(c)] ${\mathcal L}_\beta(x_k,y_k,\lambda_k)-{\mathcal L}_\beta(x_k,y_k,\lambda_{k-1})=[1/(\theta\beta)]\|\Delta \lambda_k\|^2.$
\end{itemize}
\end{lemma}
\begin{proof}
(a) In view of \eqref{sub:x}, we have  ${\mathcal L}_\beta(x_k,y_{k-1},\lambda_{k-1}) + \|x_k-x_{k-1}\|^2_G/2 \le {\mathcal L}_\beta(x_{k-1},y_{k-1},\lambda_{k-1})$, which, combined with the identity $\Delta x_k =x_k-x_{k-1}$, proves   (a).


(b)
Observe that the objective function of \eqref{sub:y} has the form
\begin{equation} \label{eq:appp1}
{\mathcal L}_\beta(x_k,\cdot,\lambda_{k-1}) + \frac{\tau}{2} \|\cdot-y_{k-1}\|^2 = (g + q)(\cdot)
\end{equation}
where $q$ is a quadratic function whose Hessian
is $Q=\beta B^*B + \tau I$. Since
$Q-m I \succ 0$ in view of \eqref{assump:betatau}, and
condition {\bf (A3)} implies that $g$ is a proper lower semi-continuous such that $g+m \|\cdot\|^2$
is convex, it follows from inequality \eqref{eq:fr45} of Lemma \ref{lem:strong_conv}  with  $y=y_{k-1}$ and $\bar y=y_k$ that
\[
(g+q)(y_{k-1}) \geq
(g+q)( y_k) + \frac{\beta}{2}\|B(y_{k-1}-{y_k})\|^2+ \frac{\tau}{2}\|y_{k-1}-{y_k}\|^2 -\frac{m}{2}\|y_{k-1}-{y_k}\|^2,
\]
which together with \eqref{eq:appp1} yields
\[{\mathcal L}_\beta(x_k,y_k,\lambda_{k-1})- {\mathcal L}_\beta(x_k,y_{k-1},\lambda_{k-1})\leq (m/2)\|\Delta y_k\|^2-(\beta/2)\|B\Delta y_k\|^2-\tau \|\Delta y_k\|^2.\]
Therefore, item (b) follows now from the fact that $\|B\Delta y_k\|^2\geq\sigma_B\|\Delta y_k\|^2$ and simple calculus.


(c) This statement follows from \eqref{eq:lambda}, the identity $\Delta \lambda_k =\lambda_k-\lambda_{k-1}$  and the fact that
\eqref{lagrangian2} implies that
$$
{\mathcal L}_\beta( x_k, y_k,\lambda_k)={\mathcal L}_\beta( x_k, y_k,\lambda_{k-1})- \inner{ \lambda_k-\lambda_{k-1}}{Ax_k+By_k-b}. \qedhere
$$
\end{proof}

Our goal now is to show that
a certain sequence associated with $\{{\mathcal L}_\beta (x_k,y_k,\lambda_k)\}$ is monotonically decreasing, namely,
the sequence  $\{ \Delta^k_\beta+\eta_k\}$ where 
\begin{align}
 \Delta^k_\beta &:= {\mathcal L}_\beta (x_k,y_k,\lambda_k) -\bar{\mathcal L} \qquad \forall k\geq 0, \label{def:Lhat}\\[2mm]
\eta_k &:=\frac{c_1}{2}\|B^*\Delta\lambda_k\|^2+\left(\frac{\beta\sigma_B+\tau-m}{4}\right)\|\Delta y_k\|^2 \qquad \forall k\geq 1,\label{def:eta}
\end{align}
and ${\bar{\mathcal L}}$,   $\eta_0 =\eta_0(y_0,\lambda_0;\theta)$ and $c_1$ are  as defined in {\bf (A4)}, \eqref{def:eta_00} and \eqref{def:c1}, respectively.

Before establishing the monotonicity property of the above sequence, we state three technical results.
The first one describes an upper bound on $ \Delta^k_\beta - \Delta^{k-1}_\beta$ in terms of three quantities
related to $\{\Delta x_k\}$, $\{\Delta \lambda_k\}$ and $\{\Delta y_k\}$, respectively.

\begin{lemma}
For every $k\geq 1$,
\begin{equation}\label{des:Ldec}
 \Delta^k_\beta+\eta_k - (\Delta^{k-1}_\beta+\eta_{k-1}) \le
-\frac{1}{2}\|\Delta x_k\|_{G}^2+ \Theta_k^1 + \Theta_k^2
\end{equation}
where
\begin{equation}\label{def:theta1}
\Theta_k^1 := 
\frac1{\beta\theta}\|\Delta \lambda_k\|^2 + \frac{c_1}2 \left(  \|B^*\Delta\lambda_k\|^2 - \|B^*\Delta\lambda_{k-1}\|^2 \right) 
\end{equation}
and
\begin{equation}\label{def:theta2}
\Theta_k^2 := -\left(\frac{\beta\sigma_B+\tau-m}{4}\right)\left(\|\Delta y_k\|^2+\|\Delta y_{k-1}\|^2\right)
\end{equation}
where $c_1$ is defined in \eqref{def:c1}.
\end{lemma}

\begin{proof}
The proof of the lemma follows by adding the three inequalities given in statements (a), (b) and (c) of Lemma~\ref{lem:auxdecL} and
using the definitions of $\Delta^k_\beta$ and $\eta_k$ in \eqref{def:Lhat} and \eqref{def:eta}, respectively.
\end{proof}

The next two  results combined provide an upper bound for  $\Theta_k^1$ in terms of  $\{\Delta y_k\}$.

\begin{lemma}\label{lem:des_theta1uk}
Let $u_k$ and $\Theta_k^1$ be as in \eqref{def:uk} and \eqref{def:theta1}, respectively. Then,
$$
\Theta_k^1 \leq \frac{\gamma}{\beta\sigma_B^+}\|u_k\|^2
$$
where $\gamma$ is defined in \eqref{def:gamma}.
\end{lemma}

\begin{proof}
Assumption {\bf (A1)} clearly  implies that $\Delta \lambda_k=-\beta\theta(Ax_k+By_k-b)\in  {\rm Im}(B)$.
Hence, it follows from  Lemma~\ref{le:147} that
$$
\|\Delta \lambda_k\|=\|\cP_B(\Delta \lambda_k)\|\leq \frac{1}{\sqrt{\sigma_B^+}}\|B^*\Delta \lambda_k\|
$$
where $\cP_B(\cdot)$ is defined in Subsection~\ref{sec:bas}.
Hence, in view of \eqref{eq:B*lambda} and \eqref{def:theta1}, we have
\begin{align*}
\Theta_k^1 &\leq \frac{1}{\beta\theta\sigma_B^+}\|B^*\Delta\lambda_k\|^2+\frac{c_1}{2}(\|B^*\Delta\lambda_k\|^2-\|B^*\Delta\lambda_{k-1}\|^2)\\
&=\left(\frac{1}{\beta\theta\sigma_B^+}+\frac{c_1}{2}\right)\|(1-\theta)B^*\Delta\lambda_{k-1}+\theta u_k\|^2-\frac{c_1}{2}\|B^*\Delta\lambda_{k-1}\|^2.
\end{align*}
Note that if $\theta=1$, then \eqref{def:c1} implies that $c_1=0$ and the above inequality implies the conclusion of the lemma.
We will now establish the conclusion of the lemma for the case in which $\theta \ne 1$.
The previous inequality together with the relation $\|s_1+s_2\|^2\leq (1+t)\|s_1\|^2+(1+1/t)\|s_2\|^2$
which holds for every $s_1,s_2\in \R^l$ and $t>0$
yield
\small{
\begin{align*}
\Theta_k^1&\leq\left(\frac{1}{\beta\theta\sigma_B^+}+\frac{c_1}{2}\right)\left[(1+t)(\theta-1)^2\|B^*\Delta\lambda_{k-1}\|^2+\left(1+\frac{1}{t}\right)\theta^2\|u_k\|^2\right]-\frac{c_1}{2}\|B^*\Delta\lambda_{k-1}\|^2\\
&=\left[\left(\frac{1}{\beta\theta\sigma_B^+}+\frac{c_1}{2}\right)(1+t)(\theta-1)^2-\frac{c_1}{2}\right]\|B^*\Delta\lambda_{k-1}\|^2+\left(\frac{1}{\beta\theta\sigma_B^+}+\frac{c_1}{2}\right)\left(1+\frac{1}{t}\right)\theta^2\|u_k\|^2\\
&=\tiny\left\{\frac{(1+t)(\theta-1)^2}{\beta\theta\sigma_B^+}-\left[1-(1+t)(\theta-1)^2\right]\frac{c_1}{2}\right\}
\|B^*\Delta\lambda_{k-1}\|^2+\left(\frac{1}{\beta\theta\sigma_B^+}+\frac{c_1}{2}\right)\left(1+\frac{1}{t}\right)\theta^2\|u_k\|^2.
\end{align*}
}
Using the above expression with $t= -1+1/|\theta-1|$ and noting that $t>0$ in view of the assumption that $\theta \in (0,2)$, we
conclude that
\begin{align*}
\Theta_k^1 &\leq \left[\frac{1}{\beta\theta\sigma_B^+}|\theta-1|-\left(1-|\theta-1| \right)\frac{c_1}{2}\right]\|B^*\Delta\lambda_{k-1}\|^2
+\left(\frac{1}{\beta\theta\sigma_B^+}+\frac{c_1}{2}\right)\frac{\theta^2}{1-|\theta-1|}\|u_k\|^2 \\
&=\frac{1}{\beta\theta\sigma_B^+}\left(1+\frac{|\theta-1|}{1-|\theta-1|}\right)\frac{\theta^2}{1-|\theta-1|}\|u_k\|^2
\end{align*}
where the last equality is due to \eqref{def:c1}.
Hence, in view of \eqref{def:gamma}, the conclusion of the lemma follows.
\end{proof}
\begin{lemma}\label{lem:desuk} The vector $u_k$ defined  in \eqref{def:uk} satisfies
\[
\|u_k\|^2 \leq 3(L^2+\tau^2)(\|\Delta y_k\|^2+\|\Delta y_{k-1}\|^2).
\]
\end{lemma}
\begin{proof}
Noting that \eqref{eq:B*lambda} implies that $u_k\in {\rm Im}\, B^*$ and using assumption {\bf (A2)} and non-expansiveness of
the projection operator, we obtain
\begin{align}\nonumber
\|u_k\|^2 &=\|\cP_{B^*}(u_k)\|^2=\|\cP_{B^*}\left(\nabla g(y_k)-\nabla g(y_{k-1})\right)+\tau \cP_{B^*}\left(\Delta y_k-\Delta y_{k-1}\right)\|^2\\\nonumber
              &\leq \left[ L\|\Delta y_k\|+{\tau}{}\|\Delta y_k -\Delta y_{k-1}\|\right]^2\\
              &\leq 3L^2\|\Delta y_k\|^2+ {3\tau^2}(\|\Delta y_k\|^2+\|\Delta y_{k-1}\|^2)\label{a34:2}
\end{align}
where  the last inequality follows from the triangule inequality and 
the relation $(s_1+s_2+s_3)^2\leq 3s_1^2+3s_2^2+3s_3^2$ for $s_1,s_2, s_3\in \R$. Therefore, the desired inequality follows trivially from \eqref{a34:2}.
\end{proof}

Finally, the next proposition shows that  the sequence $\{ \Delta^k_\beta \}$ decreases.

\begin{proposition}\label{prop:decL} 
The sequence $\{(x_k,y_k,\lambda_k)\}$ generated by the proximal ADMM satisfies
\[
 \Delta^k_\beta+\eta_k - (\Delta^{k-1}_\beta+\eta_{k-1}) \le  -\frac{1}{2}\|\Delta x_k\|_{G}^2-\delta_1  (\|\Delta y_k\|^2+\|\Delta y_{k-1}\|^2)  \quad \forall k \ge 1
\]
where $\delta_1$, $ \Delta^k_\beta$ and $\eta_k$  are as in  \eqref{assump:betatau},  \eqref{def:Lhat} and \eqref{def:eta}, respectively.
\end{proposition}
\proof 
It follows from Lemmas~\ref{lem:des_theta1uk} and \ref{lem:desuk} that
$$
\Theta_k^1\leq\frac{3\gamma(L^2+\tau^2)}{\beta\sigma_B^+}(\|\Delta y_k\|^2+\|\Delta y_{k-1}\|^2)
$$
and hence, in view of \eqref{assump:betatau} and \eqref{def:theta2}, we have
\begin{align*}
\Theta_k^1+\Theta_k^2&\leq \left(\frac{3\gamma(L^2+\tau^2)}{\beta\sigma_B^+}+\frac{m-\beta\sigma_B-\tau}{4}\right)(\|\Delta y_k\|^2+\|\Delta y_{k-1}\|^2)\\
                                      &=-\delta_1  (\|\Delta y_k\|^2+\|\Delta y_{k-1}\|^2)                                
\end{align*}
%
where the last inequality is due to the definition of $\delta_1$ in \eqref{assump:betatau}. Hence, the result follows due to~\eqref{des:Ldec}.
\endproof

The next three lemmas show how to obtain convergence rate bounds for the quantities
$\|\Delta x_j\|_{G}$, $\|\Delta y_j\|$ and  $\|\Delta \lambda_j\|$ with the aid of Proposition \ref{prop:decL}.
The first one shows that $\{  \Delta^k_\beta +\eta_k\}$ is nonnegative.

\begin{lemma}\label{lem:infL} Let  $ \Delta^k_\beta$ and $\eta_k$  be   as in \eqref{def:Lhat} and \eqref{def:eta}, respectively. Then, 
\begin{equation}\label{eq:209}
\Delta^k_\beta +\eta_k\geq 0 \quad \forall k \geq 0.
\end{equation}
\end{lemma}
\begin{proof} Let us  first consider that case  $k\geq 1$.  Assume for contradiction that there exists an index $k_0 \ge 0$ such that $     \Delta^{k_0+1}_\beta  +\eta_{k_0+1} <0$.
Since $\{ \Delta^k_\beta+\eta_k\}$ is decreasing (see Proposition~\ref{prop:decL}), we  obtain
\[
\sum_{k=1}^j  (\Delta^k_\beta+\eta_k) \le
\sum_{k=1}^{k_0}  (\Delta^k_\beta+\eta_k) + (j-k_0) (\Delta^{k_0+1}_\beta+\eta_{k_0+1})   \quad \forall j > k_0
\]
and hence 
\[
\lim_{j \to \infty} \sum_{k=1}^j  (\Delta^k_\beta +\eta_k)= - \infty.
\]
On the other hand, since $\beta\geq \bar{\beta}$, it follows from \eqref{lagrangian2}, \eqref{eq:lambda}, \eqref{def:Lhat},  \eqref{def:eta} and assumption {\bf(A4)}  that
\begin{align*}
\Delta^k_\beta+\eta_k& =\mathcal{L}_\beta(x_k,y_k,\lambda_k)-\bar{\mathcal L}+\eta_k\ge \mathcal{L}_\beta(x_k,y_k,\lambda_k)-\bar{\mathcal L} \ge
\mathcal{L}_{\bar \beta}(x_k,y_k,\lambda_k)-\bar{\mathcal L} \\
&= f(x_k)+g(y_k)+ \frac{\bar \beta}{2} \| Ax_k+ By_k -b \|^2-\bar{\mathcal L} + \frac{1}{\beta\theta}\inner{ \lambda_k}{\lambda_k-\lambda_{k-1}}\\[2mm]
&\geq \frac{1}{2\beta\theta}\left(\|\lambda_k\|^2-\|\lambda_{k-1}\|^2+\|\lambda_k-\lambda_{k-1}\|^2\right)\geq \frac{1}{2\beta\theta}\left(\|\lambda_k\|^2-\|\lambda_{k-1}\|^2\right)
\end{align*}
and hence that
\[
\sum_{k=1}^j (\Delta^k_\beta+\eta_k) \ge \frac{1}{2\beta\theta}\left(\|\lambda_j\|^2-\|\lambda_0\|^2\right)\geq -\frac{1}{2\beta\theta}\|\lambda_0\|^2 \quad \forall j \ge 1,
\]
which yields the desired contradiction. Therefore,  \eqref{eq:209} holds for $k\geq 1$. 
 Now, for the case $k=0$, the desired inequality follows from the last conclusion and Proposition \ref{prop:decL}  with $k=1$. 
\end{proof}
\begin{lemma}\label{lem:sumDeltaxylambda} For every $k \ge 1$, we have
\begin{equation}\label{des:deltay&x}
\sum_{j=1}^k\left( \frac{1}{2}\|\Delta x_j\|_{G}^2+\delta_1\|\Delta y_j\|^2 +\delta_2\|\Delta \lambda_j\|^2\right) \leq 3\max\{\Delta^0_\beta,\eta_0\}
\end{equation}
where  $\delta_1$, $ \Delta^0_\beta$ and $\delta_2$  are as defined in \eqref{assump:betatau}, \eqref{def:Ltilde} and  \eqref{def:deltas}, respectively.
\end{lemma}
\begin{proof}
First note that Proposition~\ref{prop:decL} together with Lemma~\ref{lem:infL}  yields, for every $k\geq 1$, 
\begin{equation}\label{eq:auxdeltaxy}
\sum_{j=1}^k \left( \frac{1}{2}\|\Delta x_j\|_{G}^2+\delta_1 (\|\Delta y_j\|^2+\|\Delta y_{j-1}\|^2) \right)
\leq   \Delta^0_\beta+\eta_0 \leq   2\max\{\Delta^0_\beta,\eta_0\}
\end{equation}
which, in particular, implies that
\begin{equation}\label{eq:auxdeltayBy}
 \sum_{j=1}^k (\|\Delta y_j\|^2+\|\Delta y_{j-1}\|^2) 
\leq \frac{2\max\{\Delta^0_\beta,\eta_0\}}{\delta_1}.
\end{equation}
Due to \eqref{eq:auxdeltaxy}, in order to prove  \eqref{des:deltay&x},  it suffices to show that  
\begin{equation}\label{des:deltaLambda}
\sum_{j=1}^k \|\Delta \lambda_j\|^2 \leq\frac{\max\{\Delta^0_\beta,\eta_0\}}{\delta_2}.
\end{equation}
Then, in the remaining part of the proof we will show that  \eqref{des:deltaLambda} holds. By rewriting   \eqref{def:theta1}, we have
$$
\|\Delta \lambda_k\|^2= \beta\theta\left[ \frac{c_1}2 \left(  \|B^*\Delta\lambda_{k-1}\|^2 - \|B^*\Delta\lambda_{k}\|^2\right) + \Theta_k^1 \right]  \qquad \forall k\geq1,
$$
where $\Delta \lambda_0$ is such that the pair $(\Delta y_0,\Delta \lambda_0 )$ is a solution of \eqref{def:eta_00}.
Hence, using  \eqref{def:eta_00} and Lemmas \ref{lem:des_theta1uk} and  \ref{lem:desuk},  we obtain
\begin{align*}
\sum_{j=1}^k \|\Delta \lambda_j\|^2 &\leq \beta\theta \left[\frac{c_1}{2}\|B^*\Delta\lambda_0\|^2+\sum_{j=1}^k\Theta_j^1\right] \leq \beta\theta \eta_0+\frac{\theta\gamma}{\sigma_B^+}\sum_{j=1}^k\|u_j\|^2\\
                                                           &\leq \beta\theta\eta_0+\frac{3\theta\gamma(L^2+\tau^2)}{\sigma_B^+}
                                                           \sum_{j=1}^k(\|\Delta y_j\|^2+\|\Delta y_{j-1}\|^2)\\
                                                          &\leq {\beta\theta\eta_0}+\frac{6\theta\gamma(L^2+\tau^2){ \max\{\Delta^0_\beta,\eta_0\}}}{\sigma_B^+\delta_1}
\end{align*}
where the last inequality is due to \eqref{eq:auxdeltayBy}. 
Hence, \eqref{des:deltaLambda}  follows from  the last inequality and the definition of $\delta_2$ in \eqref{def:deltas}.
\end{proof}
\begin{lemma}\label{lem:bound_deltaxylambda} For every $k\geq 1$, there exists $j\leq k$ such that
\begin{equation*}
\|\Delta x_j\|_{G}\leq \sqrt{\frac{6 \max\{\eta_0,\Delta^0_\beta\}}{k}},\quad \|\Delta y_j\|\leq \sqrt{\frac{3 \max\{\eta_0,\Delta^0_\beta\}}{\delta_1k}},\quad
\|\Delta \lambda_j\|\leq \sqrt{\frac{3\max\{\eta_0,\Delta^0_\beta\}}{\delta_2k}}
\end{equation*}
where  $\delta_1$, $\eta_0$, $ \Delta^0_\beta$ and $\delta_2$  are as defined in \eqref{assump:betatau}, \eqref{def:eta_00}, \eqref{def:Ltilde} and  \eqref{def:deltas}, respectively.
\end{lemma}
\begin{proof}
The proof of this result follows directly from Lemma~\ref{lem:sumDeltaxylambda}.
\end{proof}

We are now ready to prove Theorem~\ref{maintheo}.

\noindent {\bf Proof of Theorem~\ref{maintheo}}:
First note that the inclusion \eqref{eq:optimi.f} follows immediately from \eqref{aux.0}. Also, we obtain from \eqref{aux.212} and \eqref{aux.32} that
$$
\nabla g(y_k)-B^*\hat  \lambda_k=-(\beta B^*B +\tau)\Delta y_k,\quad Ax_k+By_k-b= -\frac{1}{\beta\theta}\Delta \lambda_k, \quad \forall k\geq 1.
$$ 
Hence, to end the proof, just combine the above identities with Lemma~\ref{lem:bound_deltaxylambda}. \hfill{ $\square$}
\appendix
\section{Auxiliary Results}
This section presents some auxiliary results which are used in our presentation.

\begin{lemma}\label{lem:strong_conv}  Assume that, for some $m \ge 0$, $g:\R^{p} \to [-\infty,\infty]$ is 
a proper lower semi-continuous function such that
$g(\cdot)+ m \|\cdot\|^2/2$ is convex and that $q(\cdot)$ is a quadratic function whose Hessian $Q \in \R^{p\times p}$   satisfies
$Q-mI \succ 0$. Then, the problem
\begin{equation}\label{es:98}
\min \{ (g+q)(y) : y \in \R^p \}
\end{equation}
has a unique optimal solution $\bar y$ and
\begin{equation}\label{eq:fr45}
(g+q)(y) \geq (g+q)(\bar y) + \frac{1}{2}\|y-\bar{y}\|_Q^2 -\frac{m}{2}\|y-\bar{y}\|^2 \quad \forall y \in \R^p.
\end{equation}
\end{lemma}
 \begin{proof} 
 Define $\tilde g :=g+ m \|\cdot\|^2/2$,  $\tilde q = q - m \|\cdot\|^2/2$ and $\tilde Q = Q - m I$.
Clearly, $\tilde g$ is a proper lower semi-continuous convex function and $\tilde q$ is a strongly convex quadratic function
whose Hessian is $\tilde Q \succ 0$. Since $g+ q=\tilde g+\tilde q$, we conclude that  the objective function of \eqref{es:98} is strongly convex, and hence
that the first statement  of the lemma follows. Moreover, we have
\[
0 \in \partial ( g+  q) (\bar y)=\partial (\tilde g+ \tilde q) (\bar y) = \partial \tilde g(\bar y) + \nabla \tilde q(\bar y)
\]
and hence
 \[
\tilde g(y) \ge  \tilde g(\bar y) - \inner{\nabla \tilde q(\bar y)}{y-\bar y}  \quad \forall y \in \R^p.
\]
On the other hand, the fact that $\tilde q$ is a quadratic function implies that
\[
\tilde q(y) = \tilde q(\bar y) + \inner{\nabla \tilde q(\bar y)}{y-\bar y} + \frac12 \| y-\bar y\|^2_{\tilde Q} \quad \forall y \in \R^p.
\]
Adding the above two relations, and using the fact that $g+q=\tilde g + \tilde q$ and the definition of $\tilde Q$, we conclude that
\eqref{eq:fr45} holds.
\end{proof}

\begin{lemma}\label{le:147} 
Let $S \in \R^{n\times p}$ be a non-zero matrix and let $\sigma^+_S$ denote  the smallest positive eigenvalue of $SS^*$. Then, 
for every $u \in \mathbb{R}^p$, there holds
\[
\|\cP_{S^*}(u)\|\leq \frac{1}{\sqrt{\sigma^+_S}}\|Su\|.
\]
\end{lemma}

\proof
Let $r$ denote the rank of $S$ and let $S=R \Lambda Q^*$ be a partial  singular-value decomposition of $S$ where
$R \in \R^{n\times r}$ is such that
$R^*R=I$, $Q \in \R^{p \times r}$ is such that $Q^*Q=I$ and $\Lambda \in \R^{r \times r}$ is a positive diagonal matrix.
It is easy to see that
\begin{equation}\label{eq:q3}
\|\cP_{S ^*}(u)\|=\|\cP_{Q}(u)\|=\|Q(Q^*Q)^{-1}Q^*u\|=\|Q^*u\| \quad \forall u \in \mathbb{R}^p.
\end{equation}
Moreover, we have
\[
\|Q^*u\|=\|\Lambda^{-1} \Lambda Q u\|\leq \|\Lambda^{-1}\|\| \Lambda Q^* u\|=\|\Lambda^{-1}\|\| R \Lambda Q^* u\|=
\|\Lambda^{-1}\|\| Su\| \quad \forall u \in \mathbb{R}^p.
\]
The result now follows from the above two relations and the fact that $\|\Lambda^{-1}\| =1/\sqrt{\sigma^+_S}$.
\endproof


\def\cprime{$'$}

\end{document}